\documentclass[12pt,twoside]{article}
\usepackage[left=1.5cm,right=1.5cm,top=3cm,bottom=2cm]{geometry}
\usepackage{paralist}
\usepackage{marginnote}
\usepackage{amsfonts}
\usepackage{enumitem}
\usepackage{tikz}
\usetikzlibrary{matrix}
\usepackage{amsthm}
\usepackage{amsrefs}
\usepackage{amssymb}
\usepackage{amsmath}
\usepackage{graphicx}
\usepackage{varwidth}
\usepackage{hyperref}
\usepackage{fancyhdr}
\usepackage{stmaryrd}
\usepackage{multirow}
\usepackage{comment}

\usepackage{blkarray}
\usepackage[autostyle]{csquotes}
\usepackage[mathscr]{euscript}
\hypersetup{
    colorlinks=true,
    linkcolor=black,
    filecolor=magenta,      
    urlcolor=black,
    citecolor=black
}
\DeclareMathAlphabet{\mathcalligra}{T1}{calligra}{c}{h}
\providecommand{\U}[1]{\protect\rule{.1in}{.1in}}
\newtheorem{theorem}{Theorem}[section]
\newtheorem{proposition}[theorem]{Proposition}

\newtheorem{corollary}[theorem]{Corollary}

\newtheorem{remark}[theorem]{Remark}
\let\oldremark\remark
\renewcommand{\remark}{\oldremark\normalfont}

\newtheorem{example}[theorem]{Example}
\let\oldexample\example
\renewcommand{\example}{\oldexample\normalfont}

\newtheorem{examples}[theorem]{Examples}
\let\oldexamples\examples
\renewcommand{\examples}{\oldexamples\normalfont}

\def\<{{\langle}}
\def\>{{\rangle}}

\def\bea{\begin{eqnarray*} }
\def\eea{\end{eqnarray*} }
\def\be{\begin{equation} }
\def\ee{\end{equation} }

\def\qed{\ifhmode\unskip\nobreak\fi\ifmmode\ifinner
\else\hskip5 pt \fi\fi\hbox{\hskip5 pt \vrule width4 pt
height6 pt  depth1.5 pt \hskip 1pt }}

\DeclareMathOperator*{\supp}{supp}

\DeclareMathOperator*{\grad}{grad}

\begin{document}

\title{Positive harmonic functions on groups and covering spaces}
\author{Panagiotis Polymerakis}
\date{}

\maketitle

\renewcommand{\thefootnote}{\fnsymbol{footnote}}
\footnotetext{\emph{Date:} \today} 
\renewcommand{\thefootnote}{\arabic{footnote}}

\renewcommand{\thefootnote}{\fnsymbol{footnote}}
\renewcommand{\thefootnote}{\arabic{footnote}}

\renewcommand{\thefootnote}{\fnsymbol{footnote}}
\renewcommand{\thefootnote}{\arabic{footnote}}

\begin{abstract}
We show that if $p \colon M \to N$ is a normal Riemannian covering, with $N$ closed, and $M$ has exponential volume growth, then there are non-constant, positive harmonic functions on $M$. This was conjectured by Lyons and Sullivan in \cite{LS}.
\end{abstract}

\section{Introduction}

An interesting problem in Riemannian geometry is the investigation of relations between
the fundamental group of a closed manifold and the geometry of its universal covering
space. According to a seminal result of Milnor \cite{M1}, the growth rate of the fundamental group and the volume growth rate of the universal covering space coincide. Another connection between the fundamental group and a more analytic aspect of the universal covering space has been established by Brooks \cite{Brooks}. He showed that the fundamental group is amenable if and only if the bottom of the spectrum of the Laplacian on the universal covering space is zero.

In this direction, Lyons and Sullivan \cite{LS} worked on the Liouville and the strong
Liouville property on the covering space; that is, the existence of non-constant, bounded
or positive harmonic functions on the covering space. It should be emphasized that it
is not known if these properties depend only on the topology of the base manifold, or if
the Riemannian metric plays a role. Following earlier work of Furstenberg, in \cite{LS}, they constructed a discretization of the Brownian motion on the covering space. Their method was modified and extended in \cite{BL, BP}. In particular, according to \cite[Theorems A, C]{BP}, the cones of positive harmonic functions, and the spaces of bounded harmonic functions, respectively, on the covering space and the group (with respect to a symmetric probability measure, whose support is the whole group) are isomorphic. Therefore, it suffices to study the validity of the Liouville and the strong Liouville property on groups. Although one would expect this problem to be simpler, it is quite complicated and these properties are far from being comprehended completely. However, the Lyons-Sullivan discretization turned out to be quite fruitful.

To set the stage, let $p \colon M \to N$ be a normal Riemannian covering of a closed manifold, with deck transformation group $\Gamma$. Lyons and Sullivan \cite[Theorem 3]{LS} showed that if $\Gamma$ is non-amenable, then there exist non-constant, bounded harmonic functions on $M$. The converse does not hold even if $M$ is the universal covering space of $N$ (cf. \cite[Theorem 5.2]{Erschler}). However, there are some results in the converse direction. More precisely, Kaimanovich \cite{K} proved that if $\Gamma$ has subexponential growth, or $\Gamma$ is polycyclic (that is, solvable and any subgroup of $\Gamma$ is finitely generated), then any bounded harmonic function on $M$ is constant. About the strong Liouville property, Lyons and Sullivan \cite[Theorem 1]{LS} showed that if $\Gamma$ is virtually nilpotent, then any positive harmonic function on $M$ is constant. It should be noticed that according to a celebrated result of Gromov \cite{Gromov}, a finitely generated group is virtually nilpotent if and only if it is of polynomial growth.

In \cite[p. 305]{LS}, Lyons and Sullivan conjectured that $\Gamma$ is of exponential growth if and only if $M$ admits non-constant, positive harmonic functions. This was proved in \cite{BBE,BE}, under the assumption that $\Gamma$ is linear, that is, a closed subgroup of $\text{GL}_{n}(\mathbb{R})$, for some $n \in \mathbb{N}$. It is noteworthy that linear groups have either polynomial growth (and therefore are virtually nilpotent) or exponential growth. Hence, the main point of \cite{BBE, BE} is that if $\Gamma$
is a linear group of exponential growth, then $M$ admits non-constant, positive harmonic
functions. In this paper, we show that this holds, without the assumption that $\Gamma$ is linear, in the following:

\begin{theorem}\label{theorem}
Let $p \colon M \to N$ be a normal Riemannian covering, with $N$ closed. If the deck transformation group of the covering has exponential growth, then there are non-constant, positive harmonic functions on $M$.
\end{theorem}

This theorem (and a more general version of it) follows from the Lyons-Sullivan discretization and our main result in the context of discrete groups. In this setting, it was proved in \cite{BE} that if $\Gamma$ is a linear group of exponential growth, then it admits non-constant, positive $\mu$-harmonic functions for any probability measure $\mu$ satisfying certain assumptions. Recently, Amir and Kozma \cite{AK} showed that if $\Gamma$ has exponential growth, then there are non-constant, positive $\mu$-harmonic functions on $\Gamma$ for any non-degenerate probability measure $\mu$ with finite support. It is important that this is not sufficient to establish Theorem \ref{theorem} by exploiting the Lyons-Sullivan discretization. Indeed, this result involves measures with finite support, while from the Lyons-Sullivan discretization one obtains measures whose support is the whole group. 

In this paper, we focus on the triviality of the Martin boundary rather than the validity of the strong Liouville property on the deck transformation group. This is natural for our purposes, keeping in mind that if the base manifold is recurrent, then non-triviality of the Martin boundary of the deck transformation group (independently from the validity of the strong Liouville property on it) implies the existence of non-constant, positive harmonic functions on the covering space. Our main result in the group-theoretic setting is the following:

\begin{theorem}\label{discrete theorem}
Let $\Gamma$ be a finitely generated group of exponential growth. Then the Martin boundary of $\Gamma$ with respect to any non-degenerate probability measure $\mu$ on $\Gamma$ is non-trivial.
\end{theorem}

It is worth to point out that if $\mu$ is a non-degenerate probability measure on $\Gamma$ such that the $\mu$-random walk on $\Gamma$ is transient, then the Martin boundary of $\Gamma$ with respect to $\mu$ consists of positive $\mu$-harmonic functions, provided that $\mu$ has finite support, or more generally, superexponential moment (cf. for instance \cite[Lemma 7.1]{GGPY}). Therefore, Theorem \ref{discrete theorem} yields the existence of non-constant, positive $\mu$-harmonic functions with respect to such a measure $\mu$. This evidently extends the result of \cite{AK}. It is also worth to mention that the recent (relatively to the current version of the paper) \cite[Corollary 1.10]{MST} is also a special case of Theorem \ref{discrete theorem} and its proof is verbatim the proof of Theorem \ref{discrete theorem} of the original version of the paper under the additional assumption that $\mu$ has superexponential moment. It should also be noticed that \cite[Corollary 1.10]{MST} does not seem to contribute to the study of covering spaces, since superexponential moment has not been established for measures arising from the Lyons-Sullivan discretization under any assumption (besides the trivial case where the covering space is compact), and as a matter of fact, there are examples illustrating that this does not hold in general.

It is quite evident that Theorem \ref{theorem} is more general than the result of \cite{BBE, BE}. From \cite[Theorem 3]{LS}, it remains to investigate the strong Liouville property on $M$, in the case where $\Gamma$ is amenable. According to \cite{Gu}, a linear amenable group is virtually polycyclic. Therefore, from \cite{BBE, BE}, we obtain a characterization for the strong Liouville property on $M$, if $\Gamma$ is virtually polycyclic. On the other hand, Theorem \ref{theorem} yields the following more
general characterization, if $\Gamma$ is solvable (or elementary amenable). This follows from a result of Milnor \cite{M2}, according to which any finitely generated solvable group has either polynomial or exponential growth. The corresponding statement for elementary amenable groups has been proved in \cite[Theorem 3.2]{Ch}.

\begin{corollary}
Let $p \colon M \to N$ be a normal Riemannian covering of a closed manifold, with solvable (or more generally, elementary amenable) deck transformation group $\Gamma$. Then $\Gamma$ has exponential growth if and only if there are non-constant, positive harmonic functions on $M$.
\end{corollary}

From Theorem \ref{theorem} and \cite[Theorem 1]{LS}, it is obvious that what is left open is the existence or not of non-constant, positive harmonic functions on $M$, in the case where $\Gamma$ is of intermediate growth; that is, $\Gamma$ has superpolynomial and subexponential growth. It is well known that there exist finitely generated groups of intermediate growth, which implies that there exist normal coverings of closed manifolds with such deck transformation groups.

Due to its relation with the fundamental group of the base manifold, the universal
covering space is of particular interest in results of this type. A major open problem in
group theory is the existence of finitely presentable groups of intermediate growth, and
some experts of the field believe that there are no such groups (cf. \cite{GP}). Equivalently, it is not known if there exists a closed manifold such that its universal covering space is of intermediate volume growth.

\medskip

\noindent\textbf{Acknowledgements:} I am deeply grateful to Anna Erschler for pointing out a gap in the previous version of the paper, which led to the reformulation of Theorem \ref{discrete theorem}. I would like to thank Werner Ballmann for some very helpful discussions and remarks. I would also like to thank the Max Planck Institute for Mathematics in Bonn for its support and hospitality. I am also grateful to the referee for some helpful suggestions.

\section{Positive harmonic functions on groups}

We begin with some standard facts about random walks on groups (cf. for instance \cite{MR2408585} and the references therein). Let $\Gamma$ be a finitely generated group and $\mu$ a probability measure on $\Gamma$. We assume throughout that $\mu$ is \textit{non-degenerate}; that is, the semigroup generated by $\supp \mu$ coincides with $\Gamma$. A function $f \colon \Gamma \to \mathbb{R}$ is called \textit{$\mu$-harmonic} if
\[
f(x) = \sum_{y \in \Gamma} \mu_{x}(y) f(y)
\]
for any $x \in \Gamma$, where $\mu_{x}(y) := \mu(x^{-1}y)$.

Suppose now that $\Gamma$ is not finite, or a finite extension of $\mathbb{Z}$, or a finite extension of $\mathbb{Z}^{2}$. Then the random walk induced by $\mu$ is transient, which means that the \textit{Green function}
\[
g(x,y) := \sum_{k = 0}^{\infty} \mu_{x}^{k}(y)
\]
is finite for any $x, y \in \Gamma$, where $\mu_{x}^{k}$ stands for the $k$-fold convolution of $\mu_{x}$ with itself, $k \in \mathbb{N}$, and $\mu_{x}^{0}$ is the Dirac measure $\delta_{x}$. It should be noticed that $g(hx,hy) = g(x,y)$ for any $x,y,h \in \Gamma$. The \textit{Green metric} on $\Gamma$ is defined as
\[
d_{g}(x,y) := \ln g(e,e) - \ln g(x,y)
\]
for any $x,y \in \Gamma$, where $e$ is the neutral element of $\Gamma$. It is worth to point out that since we do not assume that $\mu$ is symmetric, the Green metric may not be symmetric. For $r > 0$, consider the ball 
\[
B_{g}(r) := \{ x \in \Gamma : d_{g}(e,x) \leq r \}.
\] 
We know from \cite[Proposition 3.1]{MR2408585} that if $\Gamma$ has superpolynomial growth (with respect to the word metric), then its exponential growth with respect to the Green metric is bounded by
\begin{equation}\label{bound}
\limsup_{r \rightarrow +\infty} \frac{1}{r} \ln |B_{g}(r)| \leq 1.
\end{equation}

Given $y \in \Gamma$ the \textit{Martin kernel} with origin $y$ is defined by
\[
K_{y}(x,z) := \frac{g(x,z)}{g(y,z)}
\]
for any $x,z \in \Gamma$. Recall that any diverging sequence $(z_{n})_{n \in \mathbb{N}}$ in $\Gamma$ has a subsequence $(z_{n_{k}})_{k \in \mathbb{N}}$ such that $K_{y}( \cdot ,z_{n_{k}}) \rightarrow f$ pointwise as $k \rightarrow +\infty$, for some positive $\mu$-superharmonic function $f$ on $\Gamma$. It is evident that any such limit function $f$ satisfies $f(y)=1$.

It is easily checked that if the Martin boundary of $\Gamma$ with respect to $\mu$ is trivial (that is, it consists of only one point $\xi$), then the corresponding Martin kernel $K_y(\cdot,\xi)$ is the constant function $1$. Indeed, in this case, the Poisson boundary is also trivial (see for example \cite[Theorem 7.61]{W}), which means that the constant function $1$ is minimal and hence, corresponds to a point of the Martin boundary (cf. for instance \cite[Theorem 7.50]{W}).
Thus, the Martin boundary of $\Gamma$ with respect to $\mu$ is trivial if and only if any diverging sequence $(z_{n})_{n \in \mathbb{N}}$ in $\Gamma$ has a subsequence $(z_{n_{k}})_{k \in \mathbb{N}}$ such that $K_{y}( \cdot ,z_{n_{k}}) \rightarrow 1$ pointwise as $k \rightarrow +\infty$.

For a finite subset $S$ of $\Gamma$ and $x,y \in \Gamma$, consider the quantity
\[
G(S; x,y) := \sup_{z} |K_{y}(x,z) - 1|,
\]
where the supremum is taken over all $z \in \Gamma \smallsetminus S$.

\begin{proposition}\label{aux}
If the Martin boundary of $\Gamma$ with respect to $\mu$ is trivial, then for any exhausting sequence $(S_{n})_{n \in \mathbb{N}}$ of $\Gamma$ and $x,y \in \Gamma$, we have that $G(S_{n} ; x,y) \rightarrow 0$ as $n \rightarrow + \infty$.
\end{proposition}

\begin{proof}
Assume to the contrary that there exists an exhausting sequence $(S_{n})_{n \in \mathbb{N}}$ of $\Gamma$, $x,y \in \Gamma$, $z_{n} \in \Gamma \smallsetminus S_{n}$ and $c>0$, such that $|K_{y}(x,z_{n}) - 1| \geq c$ for any $n \in \mathbb{N}$. Since $(z_{n})_{n \in \mathbb{N}}$ is diverging, it follows that after passing to a subsequence, if necessary, we may assume that $K_{y}(\cdot,z_{n}) \rightarrow f$ pointwise as $n \rightarrow + \infty$, for some positive $\mu$-superharmonic function $f$ on $\Gamma$. Bearing in mind that $f(y) = 1$, while
\[
|f(x) - 1| = \lim_{n} |K_{y}(x,z_{n}) - 1| \geq c > 0,
\]
we conclude that $f$ is non-constant. Evidently, there is no subsequence $(z_{n_k})_{k \in \mathbb{N}}$ such that $K_y(\cdot,z_{n_k}) \rightarrow 1$ pointwise as $k\rightarrow+\infty$, which is a contradiction. \qed
\end{proof}\medskip

\noindent{\emph{Proof of Theorem \ref{discrete theorem}}:}
Assume to the contrary that there exists a non-degenerate probability measure $\mu$ on $\Gamma$ such that the Martin boundary of $\Gamma$ with respect to $\mu$ is trivial. Let $F$ be a finite, symmetric generating set of $\Gamma$, and denote by $W_{n}$ the set of words of length at most $n$ with respect to $F$, $n \in \mathbb{N}$. Taking into account that $F$ is finite, we derive from Proposition \ref{aux} that for any $\delta > 0$ there exists $n_{0} \in \mathbb{N}$ such that
\[
G(W_{n_{0}}; e, h) < \delta
\]
for any $h \in F$.

Fix $n > n_{0}$ and a word $x = h_{1} \dots h_{n}$ of length $n$ with respect to $F$, where $h_{i} \in F$ for any $1\leq i \leq n$. Setting $x_{i} := h_{1} \dots h_{i}$ and $x_{0} := e$, it is clear that
\begin{equation}\label{green}
g(e,x) = \prod_{i=0}^{n-1} \frac{g(x_{i},x)}{g(x_{i+1},x)} g(x,x) = \prod_{i=0}^{n-1} K_{x_{i+1}}(x_{i},x) g(x,x) = \prod_{i=0}^{n-1} K_{h_{i+1}}(e , x_{i}^{-1} x) g(e,e),
\end{equation}
where we used the left-invariance of the Green function.
For $0 \leq i < n - n_{0}$, it is immediate to verify that
\[
| K_{h_{i+1}}(e,x_{i}^{-1}x) - 1 | \leq G(W_{n_{0}};e,h_{i+1}) < \delta,
\]
since $x_{i}^{-1}x \notin W_{n_{0}}$ and $h_{i+1} \in F$. This yields that $K_{h_{i+1}}(e, x_{i}^{-1}x) \geq 1 - \delta \text{ for } 0 \leq i < n - n_{0}$. Consider the positive constant
\[
c := \min_{h \in F} \min_{z \in W_{n_{0}}} K_{h}(e,z).
\]
For $n - n_{0} \leq i \leq n - 1$, using that $x_{i}^{-1}x \in W_{n_{0}}$ and $h_{i+1} \in F$, it is easy to see that $K_{h_{i+1}}(e,x_{i}^{-1}x) \geq c$.
Then (\ref{green}) gives the estimate
\[
g(e,x) \geq (1 - \delta)^{n-n_{0}} c^{n_{0}} g(e,e),
\]
which implies that
\[
d_{g}(e,x) \leq -(n-n_{0}) \ln (1 - \delta) - n_{0} \ln c =: r(n)
\]
for any $n > n_{0}$ and any word $x$ of length $n$ with respect to $F$. Hence, we deduce that
\[
W_{n} \smallsetminus W_{n_{0}} \subset B_{g}(r(n))
\]
for any $n > n_{0}$. Bearing in mind this and (\ref{bound}), we conclude that
\[
\limsup_{n} \frac{1}{n} \ln |W_{n}| = \limsup_{n} \frac{1}{n} \ln |W_{n} \smallsetminus W_{n_{0}}|  \leq \limsup_{n} \frac{1}{n} \ln |B_{g}(r(n))| \leq - \ln(1 - \delta),
\]
which is a contradiction, because $\Gamma$ has exponential growth and $\delta > 0$ is arbitrary. \qed \medskip

It is worth to mention that the arguments of the preceding proof are valid also in the case where $\Gamma$ has intermediate growth or polynomial growth of degree at least three (exploiting again \cite[Proposition 3.1]{MR2408585}). However, the penultimate line of the proof does not lead to a contradiction unless $\Gamma$ has exponential growth.

\section{Applications to Riemannian coverings}

We now discuss some applications of Theorem \ref{discrete theorem} to Riemannian coverings. Let $N$ be a Riemannian manifold and denote by $p(t,x,y)$ the kernel of the minimal heat semigroup on $N$. We say that $\lambda \in \mathbb{R}$ belongs to the \textit{Green's region} of $N$ if
\[
\int_{0}^{+\infty} e^{\lambda t} p(t,x,y) dt < + \infty
\]
for some (and then any) $x \neq y$ in $N$. We know from \cite[Theorem 2.6]{S} that the Green's region is $(- \infty , \lambda_{0}(N) )$ or $(- \infty , \lambda_{0}(N)]$, where $\lambda_{0}(N)$ stands for the bottom of the spectrum of the Laplacian on $N$. In the first case, $N$ is called \textit{$\lambda_{0}(N)$-recurrent}, and positive $\lambda_{0}(N)$-harmonic functions on $N$ are constant multiples of one another (cf. \cite[Theorem 2.7]{S}). According to \cite[Theorem 2.8]{S}, if $\lambda_{0}(N)$ is an eigenvalue of the Friedrichs extension of the Laplacian on $N$, then $N$ is $\lambda_{0}(N)$-recurrent.

\begin{theorem}\label{general theorem}
Let $p \colon M \to N$ be a normal Riemannian covering of a $\lambda_{0}(N)$-recurrent manifold $N$. If the deck transformation group of the covering is finitely generated and has exponential growth, then there are positive $\lambda_{0}(N)$-harmonic functions on $M$ that do not descend to $N$.
\end{theorem}

\begin{proof}
Let $\varphi$ be a positive $\lambda_{0}(N)$-harmonic function on $N$, denote by $\tilde{\varphi}$ its lift to $M$, and identify the deck transformation group $\Gamma$ with a fiber of the covering. Consider the operators $L_{N} := \Delta - 2 \grad \ln \varphi$ and $L_{M} := \Delta - 2 \grad \ln \tilde{\varphi}$ on $N$ and $M$, respectively.
As explained in \cite[Subsection 1.4]{BP}, the assumption that $N$ is $\lambda_{0}(N)$-recurrent yields that $N$ is $L_{N}$-recurrent (in the terminology of \cite{BP}). Denoting by $\mu$ the probability measure on $\Gamma$ arising from the Lyons-Sullivan discretization with appropriate choice of LS-data in the sense of \cite{BP}, we derive from Theorem \ref{discrete theorem} that the Martin boundary of $\Gamma$ with respect to $\mu$ is non-trivial.
It follows from \cite{BP,BP2} that the Martin boundary $\partial_{L_{M}} M$ is also non-trivial. This yields that there exist positive $L_M$-harmonic functions on $M$ that do no descend to $N$. Indeed, otherwise, any positive $L_M$-harmonic function on $M$ would descend to a positive $L_N$-harmonic function on $N$, which has to be constant, $N$ being $L_N$-recurrent, and this would imply that $\partial_{L_M} M $ is trivial. Therefore, there exists a positive $L_{M}$-harmonic function $f \in C^\infty(M)$ which does not descend to $N$. It is straightforward to verify that $f \tilde{\varphi}$ is $\lambda_{0}(N)$-harmonic and does not descend to $N$.\qed
\end{proof}\medskip

A particular case of Theorem \ref{general theorem} is the following extension of Theorem \ref{theorem} that involves coverings of complete manifolds with finite volume. It should be noticed that if $N$ is a complete manifold of finite volume, then $\lambda_{0}(N) = 0$ is an eigenvalue of the Friedrichs extension of the Laplacian on $N$ (constant functions are eigenfunctions of the operator).

\begin{corollary}
Let $p \colon M \to N$ be a normal Riemannian covering of a complete manifold $N$ with finite volume. If the deck transformation group of the covering is finitely generated and has exponential growth, then $M$ admits non-constant, positive harmonic functions.
\end{corollary}

\noindent Department of Mathematics, University of Thessaly\\
3rd km Old National Road Lamia – Athens, 35100, Lamia, Greece\\
E-mail address: ppolymerakis@uth.gr

\end{document}